\documentclass[12pt,psfig, reqno]{amsart}
\usepackage{amsmath, amsthm, amsfonts, amssymb, latexsym, graphicx}
\usepackage[latin1]{inputenc}

\newtheorem{theorem}{Theorem}[section]
\newtheorem{lemma}[theorem]{Lemma}

\theoremstyle{definition}

\newtheorem{remark}[theorem]{Remark}
\numberwithin{equation}{section}

\renewcommand{\(}{\left(}
\renewcommand{\)}{\right)}
\newcommand{\oa}{\left\{}
\newcommand{\fa}{\right\}}

\renewcommand{\exp}[1]{\textrm{exp}\oa #1 \fa}

 \allowdisplaybreaks
\begin{document}

\title{Primeness and Dynamics of Some Classes of Entire Functions}

\author[K. S. Charak]{Kuldeep Singh Charak}
\address{
\begin{tabular}{lll}
&Kuldeep Singh Charak\\
&Department of Mathematics\\
&University of Jammu\\
&Jammu-180 006\\ 
&India\\
\end{tabular}}
\email{kscharak7@rediffmail.com}

\author[M. Kumar]{Manish Kumar}
\address{
\begin{tabular}{lll}
&Manish Kumar\\
&Department of Mathematics\\
&University of Jammu\\
&Jammu-180 006\\ 
&India\\
\end{tabular}}
\email{manishbarmaan@gmail.com}
\author[A. Singh]{Anil Singh}
\address{
\begin{tabular}{lll}
&Anil Singh\\
&Department of Mathematics\\
&University of Jammu\\
&Jammu-180 006\\
&India
\end{tabular}}
\email{anilmanhasfeb90@gmail.com }

\begin{abstract}
 In this paper we investigate the primeness of a class of entire functions and discuss the dynamics of a periodic member $f$ of this class with respect to a transcendental entire function $g$ that permutes with $f.$ In particular we show that the Julia sets of $f$ and $g$ are identical.
\end{abstract}
\renewcommand{\thefootnote}{\fnsymbol{footnote}}
\footnotetext{The work of first author is partially supported by Mathematical Research Impact Centric Support (MATRICS) grant, File No. MTR/2018/000446, by the Science and Engineering Research Board (SERB), Department of Science and Technology (DST), Government of India. }

\maketitle

\section{Introduction}

A meromorphic function $F$ is said to be factorizable with  $f$ and $g$ as left and right factors respectively, if it can be expressed as $F:=f\circ g$ where  $f$ is meromorphic and $g$ is entire ($g$ may be meromorphic when $f$ is rational). $F$ is said to be prime (left-prime and right prime), if every factorization of $F$ of the above form implies either $f$ is bilinear or $g$ is linear ($f$ is bilinear whenever $g$ is transcendental and $g$ is linear whenever $f$ is transcendental). When the factors $f$ and $g$ of $F$ are restricted to entire functions, the factorization is said to be in the entire sense. Accordingly, we may use the term $F$ is prime (left-prime and right prime) in the entire sense. For factorization of meromorphic functions one may refer to Gross \cite{gross} and Chuang and Yang \cite{chuang}.

\medskip

Suppose $H$ and $\alpha $ are entire functions satisfying
\begin{enumerate}
\item[(a)] $\alpha^{'}$ has at least one zero,\\
\item[(b)] $T(r,\alpha)=o(T(r,H)) \mbox{ as } r\rightarrow \infty $,\\
\item[(c)] $H^{'}$ and $\alpha^{'}$ have no common zeros.
\end{enumerate}
Now consider the class 
$$\mathcal{F}:= \{F_a(z)=H(z)+a\alpha (z): a\in \mathbb{C} \}.$$
$\mathcal{F}$ is a most general class containing both periodic as well as non-periodic prime entire functions.  Urabe \cite{urabe1} proved that the periodic entire functions of the form 
\begin{equation}
F_a(z)=H(z)+\frac{a}{2}\sin 2z,
\label{eq1}
\end{equation}
where 
$$H(z)=\sin z h(\cos z)$$
in which 
$$ h(w)=\exp {\psi\left((2w^{2}-1)^{2}\right)} \mbox{ for some entire function } \psi,$$ 
are prime for most values of $a\in \mathbb{C}$. We shall denote by $\mathcal{U}$, the class of all functions of the form (\ref {eq1}).\\
Also, Qiao \cite{Qiao} proved that the periodic entire functions of the form 
\begin{equation}
G_a(z):=\cos z\left(H(\sin z)+2a\right),
\label{eq2}
\end{equation}
where $H$ is an odd transcendental entire function  such that order of $H(\sin z)$ is finite, is prime for most values of $a\in \mathbb{C}.$ Let us denote by $\mathcal{Q}$, the class of functions of the form (\ref {eq2}).

 It is quite simple to note that the classes $\mathcal{U}$ and $\mathcal{Q}$ are subclasses of $\mathcal{F}$. Also $\mathcal{F}$ contains the class of non-periodic prime entire functions due to Singh and Charak \cite{apcharak}. Now it is natural to investigate the primeness of members of $\mathcal{F}$. The importance of this investigation lies in the fact that if $F_a\in\mathcal{F}$ happens to be prime in entire sense and $g$ is any entire function permutable with $F_a$, then Julia sets of $g$ and  $F_a$  are identical for most values of $a$; for example see  \cite{ng, noda, urabe}.

\medskip

 Let $f$ be an entire  function. We denote by $f^n$ the $nth$ iterate of $f$. By a Julia set $J(f)$ of an entire function $f$ we mean the complement $\mathbb{C}\setminus F(f)$ of Fatou set $F(f) $ of $f$ defined as 
$$F(f):=\left\{z\in\mathbb{C}:\{f^n\}\mbox{  is normal in a neighborhood of }z\right\}.$$
In this paper, we shall find out some more subclasses of $\mathcal{F}$ which are prime and study the dynamics of such subclasses with respect to any non-linear entire function permuting with a given member of these subclasses.

 We shall adopt the following notations in our discussions throughout:
\begin{enumerate}
\item $\mathcal{H}(D): \mbox{ the class of all holomorphic functions on a domain } D\subset\mathbb{C}.$
\item $\mathbb{C}^{*}:= \mathbb{C}\setminus \left\{0\right\},$ the punctured plane.
\item $\mathcal{E}_{T}: \mbox{ the class of all transcendental entire functions.}$
\end{enumerate}

\section{Factorization of some periodic subclasses of $\mathcal{F}$}
 The subclasses $\mathcal{U}\mbox{ and }\mathcal{Q}$ of $\mathcal{F}$ consist of prime periodic entire functions. We shall prove the primeness of a more general subclass of $\mathcal{F}$ containing $\mathcal{U}.$ Also, we shall investigate some interesting properties of the functions in this  general subclass as well as $\mathcal{Q}.$ Actually, the purpose of this investigation is to study the dynamics of these subclasses of $\mathcal{F}$ in the next section.  
\begin{lemma} \label{l1} Let $H$ and $\alpha$ be entire functions such that $H'$ and $\alpha'$ have no common zeros. Let $H_a(z):= H(z)+a\alpha (z)$, where $a\in \mathbb{C}$. Then there exists a countable set $E\subset \mathbb{C}$  such that for each $c\in \mathbb{C}$ and for any $ z, w \in \{z\in \mathbb{C}: H_{a}(z)=c, H_{a}^{'}(z)=0\}$, $\alpha (z)=\alpha(w)$ for each $a\notin E$.
\end{lemma}
\begin{proof}   For any critical point $z_0$ of $H_a$, we have 
	
	 $$ a = -\frac{H^{'}(z_0)}{\alpha^{'}(z_0)}$$
Define 

\begin{flalign}
 m(z):=-\frac{H^{'}(z)}{\alpha^{'}(z)}, ~~~~~ z\in \mathbb{C}.\label{eq:20}
\end{flalign}

 Set $A:= \mathbb{C} \setminus \{z\in \mathbb{C}: m'(z)=0 \mbox{ or } \infty\}.$ Then $A$ is an open subset of $\mathbb{C}$. Let $\left\{G_i: i\in \mathbb{N}\right\}$ be an open covering of $A$ such that $m|_{G_{i}} $ is univalent and each $D_{i}=m(G_{i})$ is an open disk.
Now consider  \begin{flalign}
 M(z)&:=H(z)+m(z)\alpha(z)\label{eq:21}\\
x_{i}(w)&:=(m|_{G_{i}})^{-1}(w),~ w \in D_{i}~ (i=1,2,3,\cdots)\label{eq:22}\\
y_{i}(w)&:=M(x_{i}(w)),~ w \in D_{i} ~(i=1,2,3,\cdots) \label{eq:23}\\
I&:=\{(i,j)\in \mathbb{N}\times\mathbb{N}: D_{i}\cap D_{j}\neq \phi, y_{i}\not\equiv y_{j}\mbox{ on } D_{i}\cap D_{j}\}\label{eq:24}\\
S_{ij}&:= \{w\in  D_{i}\cap D_{j}: y_{i}(w)=y_{j}(w)\}, ~ (i,j)\in I \label{eq:25}\\
E_{0}&:=\left(\bigcup_{i=1}^{\infty}D_{i}\right)\setminus\left(\left\{m(p): m^{'}(p)=0;p\in \mathbb{C} \right\}\cup \left(\cup_{(i,j)\in I}S_{ij}\right)\right).\label{eq:26}
\end{flalign} \\
Then $E=\mathbb{C}\setminus E_{0}$ is an at most countable subset of $\mathbb{C}$.\\
Using $\left(\ref{eq:20}\right)$ in $\left(\ref{eq:21}\right)$, we get
 \begin{flalign}
M^{'}(z)=H^{'}(z)+m^{'}(z)\alpha(z)+m(z)\alpha^{'}(z)=m^{'}(z)\alpha(z).\label{eq:27}
\end{flalign}
By $\left(\ref{eq:23}\right)$ with the help of $\left(\ref{eq:21} \right)$ and $\left(\ref{eq:22}\right)$, we obtain 
\begin{flalign}
 y_{i}(w)&=H(x_i(w))+m(x_i(w))\alpha(x_i(w)) \nonumber\\
      &=H(x_i(w))+w\alpha(x_i(w))\label{eq:28}
\end{flalign}
Also by $\left(\ref{eq:23}\right)$ together with $\left(\ref{eq:22}\right)$ and $ \left(\ref{eq:27}\right)$, we have
\begin{flalign} \label{eq:29}
y_{i}^{'}(w) &= M^{'}(x_i(w))\cdot x_{i}^{'}(w) \nonumber\\
             &=m^{'}(x_{i}(w))\cdot \alpha(x_i(w))\cdot x_{i}^{'}(w) \nonumber\\
             &= \alpha(x_{i}(w)).
\end{flalign} 
Since $a\in E_0$, we have $\{z\in \mathbb{C}: H'_{a}(z)=0\}=\bigcup_{i=1}^{\infty}\left\{x_{i}(a)\right\}$. This can be verified as follows. By assumption, $H^{'}$ and $\alpha^{'}$ have no common zeros. Then, since  by (\ref{eq:22}) $m(x_{i}(a))=a$ for $a\in D_{i}$, $\alpha^{'}(x_{i}(a))\neq0$ and hence $H^{'}_{a}(x_i(a))=0.$ Conversely, let $H^{'}_a(z_0)=0$. Then $H^{'}(z_0)+a\alpha^{'}(z_0)=0$ which implies that $m(z_0)=a$. Since $a\in E_0$, $m^{'}(z_0)\neq 0$ and so $x_i(a)=z_0$ for $a\in D_{i}$ for some $i$.

 Moreover, if $y_{i}(a)=y_{j}(a)$ for some $a\in E_0$, then by (\ref{eq:24}), (\ref{eq:25}) and (\ref{eq:26}), it follows that $y^{'}_{i}(a)=y^{'}_{j}(a)$. From (\ref{eq:29}), we have 

\begin{flalign} \label{eq:30}
\alpha(x_{i}(a))=\alpha(x_{j}(a)).
\end{flalign}
Hence,  there exists a countable set $E$ in $\mathbb{C}$ such that for each $c\in \mathbb{C}$, for any $z, w \in \{z\in \mathbb{C}: H_{a}(z)=c, H_{a}^{'}(z)=0\}$ such that $\alpha (z)=\alpha(w)$, provided $a\notin E$.
\end{proof}

\medskip

Considering $H_a(z):=\cos z\cdot h(\sin{z})+a\sin{z}$, where $h\in \mathcal{E}_{T}$, in Lemma \ref{l1} and redefine $E_0$ in (\ref{eq:26}) as follows: 
\begin{flalign} \label{eq:31}
E_{0}=\left(\bigcup_{i=1}^{\infty}D_{i}\right)\setminus\left(\left\{m(p): m^{'}(p)=0;p\in \mathbb{C} \right\}\cup \left(\cup_{(i,j)\in I}S_{ij}\right)\cup\left(\cup_{i=1}^{\infty}\left\{p: h(\sin (x_i(p)))=0\right\}\right)\right).
\end{flalign}
By (\ref{eq:30}) , we have
$$\sin x_i(a)=\sin x_j(a), \mbox{ and hence } \cos x_i(a)h(\sin x_i(a))=\cos x_j(a)h(\sin x_j(a)).$$
Again by (\ref{eq:31}), we obtain 
$$\cos x_i(a)=\cos x_j(a)$$
and hence we obtain:

\begin{lemma} \label{l2} Let $h\in \mathcal{E}_{T}$ such that $h(\pm 1)\neq 0.$ Put $H_a(z):=\cos z\cdot h(\sin{z})+a\sin{z}$, where $a\in \mathbb{C}$. Then, there exists a countable set $E$ of complex numbers such that any two roots $u$ and $v$ of the simultaneous equations $$H_a(z)=c,\ H^{'}_{a}(z)=0,$$ $\cos u=\cos v$ and $\sin u=\sin v$ for any constant $c\in \mathbb{C}$, provided $a\notin E.$
\end{lemma}



 Let $f(z)$ be periodic entire function of period $\lambda \neq 0$. Then for $z\in\mathbb{C}$, we shall denote by   $[z]$, the set $\{z+n\lambda:n\in\mathbb{Z}\}$ and for $A\subset \mathbb{C}$, by$[A]$, the set $\{[z]:z\in A\}.$

\begin{theorem} \label{t1} Let $H\in \mathcal{E}_{T}$ such that $H(-z)=-H(z), H^{'}(0)\neq 0,$ and the order of $H(\sin z)$ is finite. Put $H_a(z):=\cos z(H(\sin z)+2a)$, where $a\in \mathbb{C}$. Then there exists a countable set $E\subset \mathbb{C}$ such that $H_a$ satisfies the following properties for each $a\notin E$. 
\begin{itemize} 
\item[(i)] $H_a$ is prime.
\item[(ii)] $\#\left\{[z]: H^{'}_a(z)=0\right\}=\infty$
\item[(iii)] for any $z_1, z_2 \in \left\{z\in \mathbb{C}: H_a(z)=c, H^{'}_a(z)=0 \right\},$  $\cos z_1=\cos z_2$ for any $c\in \mathbb{C}$.
\item[(iv)] $H^{'}_a$ has only simple zeros.
\end{itemize}
\end{theorem}

\begin{proof} $(i)$ and $(iii)$ follow from [\cite{Qiao}, Theorem 1] and Lemma \ref{l1} respectively.

\medskip

  Define $$k(w)=\frac{w^2+1}{2w}H\left(\frac{w^{2}-1}{2iw}\right).$$ Then $k\in \mathcal{H}(\mathbb{C}^{*})$ with essential singularities at $0$ and $\infty$, and $$H_a(z)=\left(k(w)+a\left(\frac{w^2+1}{w}\right)\right)\circ e^{iz}.$$ Put $$H_{a}(z)=k_{a}(e^{iz}),$$ where $k_a(w)=k(w)+a(\frac{w^2+1}{w})$. Since $H^{'}_{a}(z)=k^{'}_a(e^{iz})\cdot i e^{iz}$, we have
\begin{equation}\label{eq:1a}
H^{'}_a(z)=0 \Leftrightarrow k^{'}_a(e^{iz})=0. 
\end{equation}
Thus  $k^{'}_a(w)=0$ if and only if 
$$\frac{w^{2}k^{'}(w)}{w^{2}-1}=-a.$$
By Picard's big theorem, we have 
\begin{equation}\label{eq:2a}
\#\left\{w\in \mathbb{C}^{*}: k^{'}_{a}(w)=0\right\}=\infty
\end{equation}
 for every $a\in \mathbb{C}$ with at most two exceptions.

 By (\ref{eq:1a}) and (\ref{eq:2a}), we get
 $$\#\left\{[z]: H^{'}_a(z)=0\right\}=\infty,$$
for every $a\in \mathbb{C}$ with at most two exceptions. This proves $(ii).$

\medskip

  To establish $(iv)$ it is enough to prove that $k^{'}_a$ has simple zeros (see (\ref{eq:1a})). Suppose that $k_{a}^{'}(w_0)=0$, $ k_{a}^{''}(w_0)=0$. Then 
 $$w_0(w^{2}_0-1)k^{''}(w_0)-2k^{'}(w_0)=0,~~~ a=\frac{-w^{2}_0 k^{'}(w_0)}{w^{2}_0-1}$$
$\it{Claim}$: $w(w^2-1)k^{''}(w)-2k^{'}(w)\not\equiv 0$ on $\mathbb{C}^{*}.$\\
Suppose that $w(w^2-1)k^{''}(w)-2k^{'}(w)\equiv 0$ on $\mathbb{C}^{*}$. Then 
$$k^{''}(w)=\frac{2}{w(w^2-1)}k^{'}(w),~~~w\in \mathbb{C}^{*}.$$
 Since $k^{'}(w)$ has no zeros at $w=\pm 1$, $k^{''}(w)$ has poles at $w=\pm 1$ and this is a contradiction. This establishes the claim. 

 Set $$E=\left\{a=\frac{-w^2k^{'}(w)}{w^2-1}: k^{''}_a(w)=0, w\neq 0\right\}.$$   
If $a\notin E$, then $\left\{w\in \mathbb{C}^{*}:k_{a}^{'}(w)=0, k_{a}^{''}(w)=0 \right\}=\phi.$ Therefore, $k^{'}_{a}$ has only simple zeros for $a\notin E$. This proves $(iv).$ 

  Adjoining the exceptions obtained in $(i), \ (ii)$ and $(iii)$ with $E$, the theorem holds for each $a\notin E$. 
\end{proof}


\begin{theorem}  \label{t3}Let $h\in \mathcal{E}_{T}$ such that $h(\pm 1)\neq 0.$ Put $H_a(z):=\cos{z}\cdot h(\sin{z})+a\sin{z}$, where $a\in \mathbb{C}$. Then the set $\{a\in \mathbb{C}: H_{a} \mbox{ is not prime in entire sense }\}$ is at most countable.
\end{theorem}

\begin{remark} Theorem \ref{t3} generalizes  Theorem 1 of Urabe \cite{urabe}.
\end{remark}
\textbf{Proof of theorem \ref{t3}}. Define $$h_a(w):=\left(\frac{w^2+1}{2w}\right)h\left(\frac{w^2-1}{2iw}\right)+a\left(\frac{w^2-1}{2iw}\right).$$ Then $h_a\in \mathcal{H}(\mathbb{C}^{*})$ with essential singularities at $0$ and $\infty.$ Let $H_a(z)=h_{a}(e^{iz})$. We can choose a countable subset $E$ of complex plane for which the assertion of Lemma \ref{l2} holds with respect to  $H_{a}$. We may assume $0\in E$. 
\medskip
By Second fundamental theorem of Nevanlinna [\cite{Hayman-1}, Theorem 2.3, p.43], we can choose $t\in (0,1)$ such that the inequalities 
\begin{equation}\label{eq:1}
\overline{N}(r,0,H^{'}_{a})\geq tm(r,h^{'}_a(e^{iz}))
\end{equation}
 and
\begin{equation}\label{eq:2}
\overline{N}(r,c,H_{a})\geq tm(r,H_a(z))
\end{equation} 
hold on a set of $r$ of infinite measure for any $c\in\mathbb{C}$.

\medskip

Suppose $H_{a}(z)=f(g(z))$, we consider the following cases one by one:

\medskip

{\it Case (i):}  When $f,\ g \in \mathcal{E}_{T}$.\\
 Since $H_{a}^{'}(z)=f^{'}(g(z))g^{'}(z)$, by (\ref{eq:1}) we find that $f^{'}$ has infinitely many zeros $\{t_{k}\}$, say. Since every solution of $g(z)=t_k$ is also a solution of the simultaneous equations
	$$H_{a}(z)=f(t_k),\ H_{a}^{'}(z)=0,$$
	by Lemma \ref{l2} it follows that all the roots of $g(z)=t_{k}$  lie on a  single straight line $l_n$. The set $\left\{l_{n}:n\in \mathbb{N}\right\}$ is infinite otherwise by Edrei's theorem \cite{edrei}, $g$ would reduce to a polynomial (of degree 2) which is not in reason. Therefore, by Theorem 3 of Kobayashi \cite{kobayashi}, we have

\begin{equation}\label{eq:4}	
		g(z)=P(e^{Az}),
\end{equation}
where $P$ is a quadratic polynomial and $A$  is a non-zero constant. Let $\left\{z_{k,j}\right\}_{j=1}^{\infty}$ be the roots of $g(z)=t_k$. Then $\left\{z_{k,j}\right\}_{j=1}^{\infty}$ are also the common roots of the simultaneous equations $$H_{a}(z)=f(t_{k}), \ H^{'}_{a}(z)=0.$$
By Lemma \ref{l2}, we have $$\cos z_{k,i}=\cos z_{k,j} \mbox{ and } \sin z_{k,i}=\sin z_{k,j}.$$ 
This implies that 
$$z_{k,i}=z_{k,j}+2\pi m_0,\mbox{ for some }m_0\in \mathbb{Z}.$$
By (\ref{eq:4}), $g$ is a periodic function with period $2\pi i/lA$, where $l=1 $ or $2$. Thus, we have $A=i/N$ for the integer $N=lm_0 $, and  $$H_a(z)=h_a(e^{iz})=f(P(e^{iz/N})).$$
Put $w=e^{iz/N}$. Then 
\begin{equation}\label{eq:E1}
h_a(w^{N})=f(P(w))\mbox{ for all } w\neq 0. 
\end{equation}
Note that the left hand side of (\ref{eq:E1}) has an essential singularity at $0$ but the right hand side is holomorphic at $0.$ This contradiction shows that {\it Case(i)} can't occur.

\medskip

 {\it Case (ii):} When $f\in \mathcal{E}_{T}$ and $g$ is a polynomial of degree at least two.\\
 By (\ref{eq:1}), $f^{'}$ has infinitely many zeros $\left\{t_k\right\}$, say. Let $p_k$ and $q_k$ be two roots of $g(z)=t_k$. Then $p_k$ and $q_k$ are also common roots of the simultaneous equations 
\begin{equation}\label{eq:3}
H_a(z)=f(t_k), \ H^{'}_a(z)=0.
\end{equation}
 By Lemma \ref{l2}, it follows that 
\begin{equation}\label{eq:41}
p_k-q_k=2m_k\pi,
\end{equation}
 for some integer $m_k$. Further, by Renyi's theorem \cite{renyi}, $g(z)=bz^2+cz+d$ for some $b\neq 0, c, d\in \mathbb{C}$ and hence 
\begin{equation}\label{eq:42}
p_k+q_k=-\frac{c}{b}.
\end{equation}
Now by (\ref{eq:41}) and (\ref{eq:42}), it follows that all $p_k$ and $q_k$ lie on a single straight line (independent of $t_k$, $k\in \mathbb{N}$).

 Since $H_{a}^{'}$ is periodic, we have $$N(r,0,H^{'}_{a})\leq N(r,0,f^{'}(g))+N(r,0,g^{'})$$
$$=N(r,0,f^{'}(g))+O(\log r)=o(m(r,h^{'}(e^{iz}))).$$ This contradicts (\ref{eq:1}) showing that {\it Case(ii)} is not possible.

\medskip

{\it Case (iii):} When $f$ is polynomial of degree $d$ $(\geq 2)$ and $g\in \mathcal{E}_{T}$.\\ 
By Renyi's Theorem \cite{renyi}, $g$ is periodic and therefore, we can express $g$ as $$g(z)=G(e^{Bz}),$$ $G\in \mathcal{H}(\mathbb{C}^{*})$ with an essential singularity at $0$ or $\infty$ and $B$ is a non-zero constant. Let $w_0$ be a zero of $f^{'}$. Then $G(z)=w_0$ has at most finitely many roots and so by Picard's big theorem it follows that  $f^{'}$ has exactly one zero, say $w_0.$ Thus we can express $f^{\prime}$ as 
 $$f^{'}(w)=b(w-w_0)^{d-1} \mbox{ and  hence } f(w)=\alpha (w-w_0)^{d}+c$$ for some constants $\alpha (\neq 0)$ and $c$. Therefore, $$H_a(z)=\alpha \(g(z)-w_0\)^d+c.$$ Hence $N(r,c,H_a)=dN(r,w_0,g).$ Since $G(w)=w_0$ has at most finitely many roots, $N(r,w_0,g)=o(m(r,H_a(z)))$ which is contrary to (\ref{eq:2}) showing that {\it Case(iii)} also fails to occur. 

\medskip

Hence $H_a(z)$ is prime in entire sense.	$\qed$

\medskip

 Using Theorem \ref{t3} and Lemma \ref{l2} and following the proofs of $(ii)$ and $(iv)$ in Theorem \ref{t1}, we obtain:
\begin{theorem} \label{t4} Let $h\in \mathcal{E}_{T}$ such that $h(\pm 1)\neq 0$. Put $H_a(z):=\cos{z} \cdot h(\sin{z})+a\sin{z}$, where $a\in \mathbb{C}$. Then there exists a countable set $E\subset \mathbb{C}$ such that $H_a$ possesses  the following properties for each $a\notin E$:
\begin{itemize} 
\item[(i)] $H_a$ is prime in entire sense;
\item[(ii)] $\#\left\{[z]: H^{'}_a(z)=0\right\}=\infty$;
\item[(iii)] for any $z_1, z_2 \in \left\{z\in \mathbb{C}: H_a(z)=c, H^{'}_a(z)=0 \right\}$, $~~\cos z_1=\cos z_2 \mbox{ and }~~\sin z_1=\sin z_2$ for any $c\in \mathbb{C}$; and
\item[(iv)] $H^{'}_a$ has only simple zeros.
\end{itemize}
\end{theorem}
\section{Dynamics of non-linear entire functions permutable with members of subclasses of $\mathcal{F}$}
The main result in this section is obtained by utilizing the argument due to Y. Noda \cite{noda}, faithfully, with certain modifications.

 Let $D_0:=D\setminus \left\{z: f^{'}(z)g^{'}(z)=0 \right\}$. For $f, g \in \mathcal{H}(D)$, define a relation on $D_0$ with respect to $f$ and $g$, denoted by $\sim_{\left(f,g\right)} $ as follows:

 Let $z, w \in D_0$. We write $z \sim_{\left(f,g\right)} w$ if and only if $f(z)=f(w), g(z)=g(w)$ and there are neighborhoods $U_z$ and $U_w$ of $z,$ and  $w$ respectively such that $f(U_z)=f(U_w)$, $g(U_z)=g(U_w)$ and $\left(f_{|U_w}\right)^{-1}of_{|U_z}=\left(g_{|U_w}\right)^{-1}\circ g_{|U_z}$ in $U_z$. Then $\sim_{\left(f,g\right)} $ is an equivalence relation on $D_0.$ 

\indent
  	
Y. Noda [\cite{noda}, Lemma 2.1] proved that for $f, \ g \in \mathcal{H}(D)$ and $z_0 \in \mathbb{C}$, there exist a neighborhood $N_{z_0}$ of $z_0, \ h\in \mathcal{H}(N_{z_0})$ and $\phi, \ \psi \in \mathcal{H}(h(N_{z_0}))$ satisfying  
\begin{enumerate}
	\item $f^{'}(z)\neq 0, g^{'}(z)\neq 0, h^{'}(z)\neq 0$  $z\in N_{z_0}\setminus \{z_0\};$
	\item $z\sim_{\left(f,g\right)}w$ if and only if $h(z)=h(w)$  $z,w \in N_{z_0}\setminus \{z_0\};$ and
	\item $f=\phi \circ h, g=\psi \circ h.$
\end{enumerate}
This information lead Y. Noda \cite{noda} to extend the above equivalence relation to $D$ as follows:
\indent  Let $z, w\in D$. We write $z\sim_{\left(f,g\right)}w$ if and only if $f(z)=f(w), g(z)=g(w)$ and there exists a conformal map $\phi$ defined in a neighborhood of $h_1(z)$ such that $\phi(h_1(z))=h_2(w)$, $\phi_1=\phi_2\circ\phi$, $\psi_1=\psi_2\circ\phi ,$ where $h_j, \phi_j, \psi_j (j=1,2)$ satisfy the conclusions of Lemma $2.1$ of Noda\cite{noda}( mentioned in the preceding discussion).

\medskip

Using this equivalence relation, Y. Noda\cite{noda}  proved the existence of the greatest common right factor of entire functions:
\begin{lemma}[\cite{noda}, p.5]\label{rightfactor} Let $f, \ g \in \mathcal{H}(\mathbb{C}).$  Then there exits $F \in \mathcal{H}(\mathbb{C})$ and $\phi, \ \psi \in \mathcal{H}(F(\mathbb{C}))$  such that $f=\phi \circ F$ and $g=\psi \circ F.$ $F(z)=F(w)$ if and only if $z\sim_{\left(f,g\right)} w.$
\end{lemma}
The entire function $F$ in Lemma \ref{rightfactor} is called the {\it greatest common right factor } of $f$ and $g$ (for a more general definition one may refer to [\cite{noda}, p.2]).

\medskip

We also require the following key lemmas for proving our result in this section:

\begin{lemma}[\cite{baker}, Satz 6] \label{baker0} Let $f\in \mathcal{E}_{T} $ such that $f$ permutes with a polynomial $g$. Then $g(z)=\omega z+\beta$ $\left(\omega=\exp{2\pi ik / p}, \ k,p\in\mathbb{N}, (k,p)=1\right)$. Further, if $\omega\neq 1$, then $f(z)=c+(z-c)F_0\left((z-c)^p\right),$ where $c=\beta/{(1-\omega)}$ and $F_0$ is an entire function.
\end{lemma}
\begin{lemma}[\cite{baker}, Satz 7]\label{baker} Let $f$ and $g$ be permutable entire functions. Then there exist a positive integer $n$ and $R_0>0$ such that $M\left(r,g\right)<M\left(r,f^n\right)$ for all $r>R_0$.
\end{lemma}
\begin{lemma}[\cite{clunie},Theorem 1]\label{clunie1} Let $f,\ g\in \mathcal{E}_{T}$. Then $$\limsup_{r\to\infty}\frac{\log{M\left(r,f\circ g\right)}}{\log{M\left(r,g\right)}}=\infty .$$
\end{lemma}
\begin{lemma}[\cite{noda}, Lemma 2.5]\label{noda} Suppose that $f, \ g\in  \mathcal{H}(\mathbb{C})$ are  permutable   and $\left(F,S\right)$ be a greatest common right  factor of $f$ and $g$. Let there be a subset $A\subset\mathbb{C}$ such that $\#f(A)=1$ and $\#g\left(A\right)=1$ and $r$ be the order of $f$ at some point of $g\(A\)$. Then there exists a subset $A'\subset A$ such that $\#F(A')=1$ and $\#A'\geq \# A/{r}$.
\end{lemma}
\begin{lemma}[\cite{noda}, Lemma 3.1]\label{urabe1} Let $f\in \mathcal{E}_{T}$ and $A$ be a discrete subset of $\mathbb{C}$ such that $\#f^{-1}\left(A\right)=\infty$. Then $\sup_{w\in A}\#\left(f^{-1}\left(\{w\}\right)\cap A^c\right)=\infty$.
\end{lemma}
\begin{lemma}[\cite{noda},Lemma 5.3]\label{noda0} Let $f$ be a periodic entire function and $A$ be a discrete subset of $\mathbb{C}$ such that $\#\left[f^{-1}\left(A\right)\right]=\infty$. Then $\sup_{w\in A}\#\left[f^{-1}\left(\{w\}\right)\cap A^c\right]=\infty$.
\end{lemma}
\begin{lemma} [\cite{noda}, Lemma 5.4] \label{l3} Let $h\in \mathcal{H}(\mathbb{C}^{*})$ satisfying that $\#\left\{w:h^{'}(w)=0\right\}=\infty$. Put $f(z)=h(e^{z})$. Let $g\in \mathcal{E}_{T}$ such that $g$ permutes with $f$. Then for each $N\in \mathbb{N}$, there exists $c$ such that $g^{'}(c)=0, \#\left\{[z]: f(z)=c, f^{'}(g(z))=0\right\}\geq N.$
\end{lemma}
\begin{lemma}[\cite{yang}, Lemma 2.1]\label{baker1} Suppose that $f,\ g\in \mathcal{E}_{T}$ such that $g(z)=af(z)+b$, where $a,b\in\mathbb{C}$. If $g$ permutes with $f$, then $J(f)=J(g)$.
\end{lemma}

 Lemma 5.1 of \cite{noda} can be straight way extended to:
\begin{lemma}\label{newperiod} Suppose that $f$ is periodic entire function with a period $\lambda\neq 0$ and $g$ be entire function permutable with $f$ and $\left(F,S\right)$ be a greatest common right  factor of $f$ and $\exp {(2\pi i/ \lambda) g}$. Let there be a subset $A\subset\mathbb{C}$ such that $\#f(A)=1$ and $\#\exp {(2\pi i/\lambda )g(A)}=1$ and $r$ be the order of $f$ at some point of $g\(A\)$. Then there exists a subset $A'\subset A$ such that $\#F(A')=1$ and $\#A'\geq \# A/{r}$.
\end{lemma}

 Now we state and prove the main result of this section:
\begin{theorem} \label{t2} Let $f$ be a periodic entire function satisfying the following properties:
\begin{itemize} 
\item[(i)] $f$ is prime in entire sense;
\item[(ii)] $\#\left\{[z]: f^{'}(z)=0\right\}=\infty ;$
\item[(iii)] the set $\left\{z\in \mathbb{C}: f(z)=c, f^{'}(z)=0 \right\}$ is distributed over a finite number of distinct straight lines for any $c\in \mathbb{C}$; and
\item[(iv)] the multiplicities of zeros of $f^{'}$ are uniformly bounded.
\end{itemize}
If $g$ is any non-linear entire function permutable with $f$, then $J(g)=J(f)$.
\end{theorem}

\begin{proof}  Suppose that $\lambda\neq 0$ is the period of $f$ and assume that $g\in \mathcal{E}_{T}$ such that $g$ permutes with $f$. By $(ii)$ and Lemma \ref{l3}, for each positive integer $N$ there exists $c$ such that 
$$g^{'}(c)=0\mbox{ and }\# \left\{[z]: f(z)=c, f^{'}(g(z))=0\right\}\geq N.$$  
Let $\mathcal{A}$ be a subset of $\mathbb{C}$ such that $[z]\neq [w]$ for $z,w\in \mathcal{A}, z\neq w$ and $$[\mathcal{A}]=\left\{[z]: f(z)=c, f^{'}(g(z))=0\right\}.$$ Since $f\circ g(\mathcal{A})=g\circ f(\mathcal{A})=\left\{g(c)\right\},$ we have 
\begin{flalign} \label{eq:a1} g(\mathcal{A})\subset \left\{z:f(z)=g(c), f^{'}(z)=0 \right\}.
\end{flalign}
 In $(iii)$, let us assume that the solutions of the simultaneous equations are distributed on $t$ straight lines, then there exists a subset $\mathcal{B}\subset \mathcal{A}$ such that $g(\mathcal{B})$ lie on a single straight line(which is parallel to the line passing through the origin and $\lambda$) with $\#\mathcal{B}\geq N/t$.

 {\it Claim 1:} $[g{\mathcal(B)}]$ is finite.\\
 Suppose that $[g(\mathcal{B})]$ is infinite. Let $X$ be the set of distinct points such that $[X]=[g(\mathcal{B})]$ and $[z]\neq[w],~~z,w\in X, z\neq w.$ We can choose the set $X$ such that all points of $X$ lie on a small line segment. Since $X$ is infinite, $X$ has an accumulation point which contradicts (\ref{eq:a1}). This establishes the claim. 

 Let $z_1, \cdots, z_p \in \mathbb{C},(p\geq 1)$ such that $$g(\mathcal{B})\subset \bigcup_{i=1}^{p}\left\{z_i+n\lambda: n\in \mathbb{Z}\right\}.$$ Therefore, we have a subset $\mathcal{C}\subset \mathcal{B} $ and $z_i, $ for some $i\in \left\{1,\cdots,p\right\}$ such that $$g(\mathcal{C})\subset \left\{z_i+n\lambda: n\in \mathbb{Z}\right\}\mbox{ and }\#\mathcal{C}\geq N/pt.$$ This implies that $\#\left(\exp{\left(2\pi i/\lambda\right) g(\mathcal{C})}\right)=1$. 

\medskip

On the other hand, $f(\mathcal{C})=\left\{c\right\}$ and thus $\#f(\mathcal{C})=1$. By Lemma \ref{rightfactor}, Lemma \ref{newperiod} and $(iv),$  there exist $F\in \mathcal{H}(\mathbb{C}), \ \phi, \psi \in \mathcal{H}(F(\mathbb{C}))$ and a subset $\mathcal{D}\subset \mathcal{C}$ such that $f=\phi \circ F, \  \exp{\left(2\pi i/\lambda\right) g}=\psi \circ F$, $\#F(\mathcal{D})=1$, $\# \mathcal{D}\geq N/(s+1)pt $, where $s$ denotes the maximum multiplicity of zeros of $f^{'}$. Since we can choose $N$ arbitrarily large, $F\in \mathcal{E}_{T}$.

 We claim that $F(\mathbb{C})=\mathbb{C}.$ For, it is enough to show that $F$ has no exceptional value on $\mathbb{C}$. Suppose on the contrary that there exists $c\in \mathbb{C}$ such that $F=c+e^{Q}$ for some entire function $Q$. Then $f(z)=\phi(c+e^{w})\circ Q(z)$. Since $f$ is prime in entire sense, $Q(z)=a_1 z+a_2(a_1\neq 0).$ Thus $f$ has a period $2\pi i/a_1.$ Since $\lambda$ is the fundamental period of $f$, $2\pi i/a_1= \lambda p$ for some integer $p$. Therefore, $a_1= 2\pi i/{\lambda p}$, $F(z)=c+\exp{(2\pi i/{\lambda p})z+a_2}.$ Since $\#F(\mathcal{D})=1$, we have $(z-w)\in \lambda\mathbb{Z}$ for all $z,w \in \mathcal{D}$. Thus $[z]=[w]$ for all $z,w \in \mathcal{D}$, a contradiction, and therefore, $F(\mathbb{C})=\mathbb{C}.$ Hence, $\phi, \ \psi \in \mathcal{H}(\mathbb{C}).$ 

\medskip
 
 Since $\exp{\left(2\pi i/\lambda\right) g}=\psi \circ F$, we have $\psi(z)\neq 0$ for $z\in \mathbb{C}.$ Therefore, $\psi=\exp{G}$ with $G\in \mathcal{H}(\mathbb{C})$ and so $\exp{\left(2\pi i/\lambda\right) g}=\exp G\circ F.$ Hence  
$$g=\left(\lambda/2\pi i\right)G\circ F+q\lambda$$
 for some $q \in \mathbb{Z}.$ 

 Put $K=\left(\lambda/2\pi i\right)G+q\lambda$. Then $g=K\circ F$. 
 
   Since $F$ is transcendental and $f$ is prime in entire sense, we see that $\phi$ is linear. Therefore, $g=K\circ \phi ^{-1}\circ f$. Put $g_1=K\circ \phi ^{-1}$. Then $g=g_1\circ f.$  Note that $f\circ g_1=g_1\circ f$. If $g_1$ is transcendental, then by the same argument there exists $g_2\in \mathcal{H}(\mathbb{C})$ such that $g=g_1\circ f=g_2\circ f{}^{2}$. Similarly we have $g=g_n\circ f^{n}(n=1,2,\cdots)$, whenever all $g_n(n=1,2,\cdots)$ are transcendental.  This gives a contradiction by Lemma  \ref{baker} and Lemma \ref{clunie1}. Thus, $g_n$ is a polynomial for some $n$. By Lemma \ref{baker0}, we find that $g_n(z)=lz+m(l\neq 0)$ and so $g=lf^{n}+m$. Hence by Lemma \ref{baker1} $J(g)=J(f)$. 
\end{proof}

\medskip

\begin{remark} \label{r1} Condition $(iii)$ in Theorem \ref{t1} as well as Theorem \ref{t4} implies that the set $\left\{z\in \mathbb{C}: H_a(z)=c, H^{'}_a(z)=0 \right\}$ is distributed over two distinct straight lines and single straight line respectively.

\end{remark}
 
\begin{remark} Let $f$ be a periodic entire function satisfying the conclusion of Theorem \ref{t1} (or Theorem \ref{t4}). Let $g\in \mathcal{E}_{T}$ such that $g$  permutes with $f$. Then by Theorem \ref{t2} and Remark \ref{r1}, $J(f)=J(g).$
\end{remark}



\section{A non-periodic subclass of $\mathcal{F}$}
Let $H(z)=P(z)\cdot F(\alpha(z)),$ where $P(z)$ is polynomial of degree $n$, and $F, \alpha \in \mathcal{H}(\mathbb{C})$  such that  $F_{a}(z):=H(z)-a \alpha (z)\in\mathcal{F}$ for any $a\in \mathbb{C}$. 
\medskip

On the similar lines of Lemma \ref{l2} we get:
\begin{lemma}\label{KSC}  Let $H(z)=P(z)\cdot F(\alpha(z)),$ where $P(z)$ is polynomial of degree $n$ , and $F, \alpha \in \mathcal{H}(\mathbb{C}).$ Put $F_{a}(z)=H(z)-a\alpha(z)$, where $a\in \mathbb{C}$. Suppose that $H^{'}$ and $\alpha^{'}$ have no common zeros. Then $\# \{z\in \mathbb{C}:F_{a}(z)=c, F_{a}^{'}(z)=0\}\leq n$, for all $c \in \mathbb{C} $, provided that $a \notin E.$
\end{lemma}

\begin{lemma}\label{Noda}\cite{ozawa1} Let $F\in \mathcal{E}_{T}$ satisfy $N\(r,0,F^{'}\)> k T\( r,F^{'}\)$ on a set of $r$ of infinite linear measure for some $k>0$. Assume that the simultaneous equations $F(z)=c, F^{'}(z)=0$ have finitely many solutions only for any constant $c$. Then $F$ is left-prime in the entire sense.
\end{lemma}
\begin{theorem}\label{npc} Let $F_a(z)= H(z)-a\cdot \alpha (z)\in\mathcal{F}$, where $a\in \mathbb{C}, \ H(z)=P(z)\cdot F(\alpha(z))$ and $P(z)$ is polynomial of degree $n$. Then there exists a countable set $E \subset \mathbb{C}$ such that $F_{a}$ satisfies the following properties for each $a\notin E$.\\
	$(a)$ $F_{a}$ is left prime in entire sense. $F_a$ happens to be prime if $P(z)$ is a polynomial of degree 1.\\
	$(b)$ $\#\{z\in \mathbb{C}: F_{a}(z)=c, F^{'}_{a}(z)=0\}\leq n$ for all $c\in \mathbb{C}$.\\
	$(c)$ $F_{a}$ has infinitely many critical points and each is of multiplicity $1.$\\
	\end{theorem}
	
\begin{remark}\label{r4} By (b) and the first half of (c) of  Theorem \ref{npc}, it follows that for each $a\notin E$, $F_a$ is non-periodic.
\end{remark}
From $(b)$ and  first half of $(c)$ in Theorem \ref{npc} we observe that $F_{a}$ can not be of the form $f\circ q$ for some periodic entire function $f$ and polynomial $q$. Suppose on the contrary that 
\begin{equation}\label{eq:15}
F_{a}(z)=f\circ q(z),
\end{equation}
for some periodic entire function $f$ with period $\lambda$ and for some polynomial $q$. Then 

\begin{equation}\label{eq:16}
F^{'}_{a}(z)=f^{'}(q(z))\cdot q^{'}(z).
\end{equation}

By first half of (c) in Theorem \ref{npc}, $f^{'}$ has infinitely many zeros. Let $f^{'}(z_0)=0$ for some $z_0 \in \mathbb{C}$. Then $f(z_{0}+n\lambda)=f(z_0)$ and $f^{'}(z_{0}+n\lambda)=0$ for all $n\in \mathbb{Z}$. Let $w_{n}\in \mathbb{C}$ be such that $q(w_n)=z_{0}+n\lambda$ for  $n\in \mathbb{Z}$. Then $F_{a}(w_n)=f(z_0)$ and $F_{a}^{'}(w_n)=0$ for all $n \in \mathbb{Z}$, which violates (b) of Theorem \ref{npc}.\\

\medskip

In fact an entire function satisfying $(b)$ and $(c)$ of Theorem \ref{npc} is not of the form $g\circ Q$, where $g$ is periodic entire function and $Q$ is a polynomial. Thus by Ng[\cite{ng}, Theorem 1] it follows that if $f\in \mathcal{E}_{T}$ satisfying $(a), \ (b),$ and $(c)$ of Theorem \ref{npc} and permuting with a non-linear entire function $g,$ then $J(f)=J(g);$ Theorem \ref{npc} provides an illustration of this conclusion.

\medskip

\textbf{Proof of Theorem \ref{npc}} Let $E_0$ be a countable subset of $\mathbb{C}$ such that the assertions of Lemma \ref{KSC} hold for $F_{a}(z)=H(z)-a\alpha(z)$ as soon as $a\in \mathbb{C}\setminus E_{0}$. Clearly, 
\begin{equation}\label{eq:11}
N(r,0,F_{a}^{'})=N\(r,a,\frac{H^{'}}{\alpha^{'}}\).
\end{equation}
By the second fundamental theorem of Nevanlinna [\cite{Hayman-1}, Theorem 2.3, p.43], it follows that for any $k\in (0,1)$, 
\begin{equation}\label{eq:12}
N\(r,a,\frac{H^{'}}{\alpha^{'}}\)> k T \(r,\frac{H^{'}}{\alpha^{'}}\)
\end{equation}
hold for every $r$ outside a set of finite linear measure and for every $a$ outside $E_1$, where $E_1$ is at most countable subset of $\mathbb{C}$. Let $E_2=E_{0}\cup E_{1}$. Then $E_2$ is an at most countable subset of $\mathbb{C}$ and (\ref{eq:12}) holds for every $a\in \mathbb{C}\setminus E_2$, and thus from (\ref{eq:11}), we have
\begin{equation}\label{eq:13}
N(r,0,F_{a}^{'})>k T\(r, \frac{H^{'}}{\alpha^{'}}\).
\end{equation}
Since $T(r,\alpha)=o\(T(r,H)\)$, (\ref{eq:13}) gives 
$$N(r,0,F_{a}^{'})>k T\(r, H^{'}\).$$
and hence 
\begin{equation}\label{eq:14}
N(r,0,F_{a}^{'})>k T\(r, F_{a}^{'}\)
\end{equation}
holds for all $r$ outside a set of finite linear measure and for all $a\in \mathbb{C}\setminus E_2$. Thus, combining Lemma \ref{KSC} with Lemma \ref{Noda}, it follows that $F_{a}(z)$ is left-prime in entire sense, for all $a\in \mathbb{C}\setminus E_2$.

   To show that $F_a$ is right prime in entire sense, when $P(z)$ is linear polynomial. Let $F_a=g\circ h$, where $g\in \mathcal{E}_{T}$ and $h$ is a polynomial of degree atleast two.  Then from (\ref{eq:14}), $g'$ has infinitely many zeros $\left\{z_n\right\}$. For all $n$ sufficiently large, $h(z)=z_n$ admits atleast two distinct roots which comes out to be the solutions of the simultaneous equations 
	$$F_a(z)=g(z_n), F_a'(z)=0,$$
	contradicting the fact that these simultaneous equations have atmost one solution. Thus $h$ is linear and hence $F_a$ is right prime in  entire sense. This shows that $F_a$ is prime in entire sense. By Remark \ref{r4} and Lemma 3.1 of \cite{chuang}, $F_a$ is prime.

\medskip

	(b) follows from Lemma \ref{KSC} whereas the first half of (c) follows from equation (\ref{eq:14}) .\\
To prove the second half of (c), suppose there is $z_0\in \mathbb{C}$ such that $F_{a}^{'}(z_0)=0, F_{a}^{''}(z_0)=0$. Then 
$$a= \frac{H^{'}(z_0)}{\alpha^{'}(z_0)},~~ \alpha^{'}(z_0)H^{''}(z_0)-\alpha^{''}(z_0)H^{'}(z_0)=0$$
{\it Claim 1 : $\alpha^{'}(z)H^{''}(z)-\alpha^{''}(z)H^{'}(z)\not\equiv0$}.\\
 Suppose on the contrary that $\alpha^{'}(z) H^{''}(z)-\alpha^{''}(z)H^{'}(z)\equiv 0.$ Then 
$$H^{''}=\frac{\alpha^{''}}{\alpha^{'}}H^{'}.$$
Since $\alpha^{'} $ has at least one zero and $ \alpha^{'}$ and $H^{'}$ have no common zero which implies that $H^{''}$ has a pole, a contradiction and this proves the claim. \\
Now it follows that the set 
$$E_{3}:=\left\{t=\frac{H^{'}(z)}{\alpha^{'}(z)}:F^{''}_{t}(z)=0 \right\}$$
is at most countable and for any $a\notin E_3,  \left\{z:\alpha^{'}(z)H^{''}(z)-\alpha^{''}(z)H^{'}=0 \right\}=\phi$. Therefore, $F^{'}_{a}(z)$ has only simple zeros for $a\notin E_3$.\\

The set $E:=E_2\cup E_3$ is at most countable and the above conclusions hold for each $a \notin E$.
$\qed$

\bibliographystyle{amsplain}

\end{document}